\documentclass[a4paper]{article}

\usepackage{amsmath}
\usepackage{amsfonts}
\usepackage{amssymb}         
\usepackage{amsopn}
\usepackage{theorem}          
\usepackage{latexsym}
\usepackage{graphicx}        
\usepackage{mathrsfs}		
\usepackage{psfrag}
\usepackage{algorithmic}
\usepackage{algorithm}
\usepackage{color}
\usepackage{verbatim}
\usepackage{url}

\newtheorem{result}{\ }[section]
\theoremstyle{changebreak}                
\newtheorem{thm}[result]{Theorem}

\newtheorem{lem}[result]{Lemma}

\newtheorem{prop}[result]{Proposition}

\newenvironment{proof}
 {{\sl Proof.}\hspace*{1 ex}}%
 {{\nopagebreak\hspace*{\fill}$\Box$\par\vspace{12pt}}}
\newcommand{\transpose}[1]{{#1}^{\top}}

\begin{document}

\thispagestyle{empty}
\begin{center} 

{\LARGE Six mathematical gems from the history of Distance Geometry}
\par \bigskip
{\sc Leo Liberti${}^{1}$, Carlile Lavor${}^2$} 
\par \bigskip
\begin{minipage}{15cm}
\begin{flushleft}
{\small
\begin{itemize}
\item[${}^1$] {\it CNRS LIX, \'Ecole Polytechnique, F-91128 Palaiseau,
  France} \\ Email:\url{liberti@lix.polytechnique.fr}
\item[${}^2$] {\it IMECC, University of Campinas, 13081-970, Campinas-SP, Brazil} \\
  Email:\url{clavor@ime.unicamp.br}
\end{itemize}
}
\end{flushleft}
\end{minipage}
\par \medskip \today
\end{center}
\par \bigskip

\begin{abstract}
  This is a partial account of the fascinating history of Distance
  Geometry. We make no claim to completeness, but we do promise a
  dazzling display of beautiful, elementary mathematics. We prove
  Heron's formula, Cauchy's theorem on the rigidity of polyhedra,
  Cayley's generalization of Heron's formula to higher dimensions,
  Menger's characterization of abstract semi-metric spaces, a result
  of G\"odel on metric spaces on the sphere, and Schoenberg's
  equivalence of distance and positive semidefinite matrices, which is
  at the basis of Multidimensional Scaling. \\
  {\bf Keywords}: Euler's conjecture, Cayley-Menger determinants, Multidimensional scaling, Euclidean Distance Matrix
\end{abstract}

\section{Introduction}
\label{s:intro}
Distance Geometry (DG) is the study of geometry with the basic entity
being distance (instead of lines, planes, circles, polyhedra, conics,
surfaces and varieties). As with most everything else, it all began
with the Greeks, specifically Heron, or Hero, of Alexandria sometime
between 150BC and 250AD, who showed how to compute the area of a
triangle given its side lengths \cite{heron}.

After a hiatus of almost two thousand years, we reach Arthur Cayley's:
the first paper of volume I of his {\it Collected Papers}, dated 1841,
is about the relationships between the distances of five points in
space \cite{cayley1841}. The gist of what he showed is that a
tetrahedron can only exist in a plane if it is flat (in fact, he
discussed the situation in one more dimension). This yields algebraic
relations on the side lengths of the tetrahedron.

Hilbert's influence on foundations and axiomatization was very strong
in the 1930s {\it Mitteleuropa} \cite{hilbert}. This pushed many
people towards axiomatizing existing mathematical theories
\cite{tarskied}. Karl Menger, a young professor of geometry at the
University of Vienna and an attendee of the Vienna Circle, proposed in
1928 a new axiomatization of metric spaces using the concept of
distance and the relation of congruence, and, using an extension of
Cayley's algebraic machinery (which is now known as Cayley-Menger
determinant), generalized Heron's theorem to compute the volume of
arbitrary $K$-dimensional simplices using their side lengths
\cite{menger28}. 

The Vienna Circle was a group of philosophers and mathematicians which
convened in Vienna's {\it Reichsrat caf\'e} around the
nineteen-thirties to discuss philosophy, mathematics and, presumably,
drink coffee. When the meetings became excessively politicized, Menger
distanced himself from it, and organized instead a seminar series,
which ran from 1929 to 1937 \cite{mengerK}. A notable name crops up in
the intersection of Menger's geometry students, the Vienna Circle
participants, and the speakers at Menger's {\it Kolloquium}: Kurt
G\"odel. Most of the papers G\"odel published in the Kolloquium's
proceedings are about logic and foundations,\footnote{The first public
  mention of G\"odel's completeness theorem \cite{goedel_compl} (which
  was also the subject of his Ph.D.~thesis) was given at the {\it
    Kolloquium} \cite[14 May~1930, p.~135]{mengerK}, just three months
  after obtaining his doctorate from the University of Vienna. As for
  his incompleteness theorem \cite{goedel_inc}, F.~Alt recalls
  \cite[Afterword]{mengerK} that G\"odel's seminar \cite[22 Jan.~1931,
    p.~168]{mengerK} appears to have been the first oral presentation
  of its proof:
\begin{quote}
  There was the unforgettable quiet after G\"odel's presentation,
  ended by what must be the understatement of the century: ``That is
  very interesting. You should publish that.'' Then a question: ``You
  use Peano's system of axioms. Will it work for other systems?''
  G\"odel, after a few seconds of thought: ``Yes, any system broad
  enough to define the field of integers.'' Olga Taussky
  (half-smiling): ``The integers do not constitute a field!'' G\"odel,
  who knew this as well as anyone, and had only spoken carelessly:
  ``Well, the\dots the\dots the domain of integrity of the integers.''
  And final relaxing laughter.
\end{quote}
The incompleteness theorem was first mentioned by G\"odel during a
meeting in K\"onigsberg, in Sept.~1930. Menger, who was travelling,
had been notified immediately: with John Von Neumann, he was one of
the first to realize the importance of G\"odel's result, and
began lecturing about it immediately \cite[Biographical
  introduction]{mengerK}.}  but two, dated 1933, are about the
geometry of distances on spheres and surfaces. The first \cite[18
  Feb.~1932, p.~198]{mengerK} answers a question posed at a previous
seminar by Laura Klanfer, and shows that a set $X$ of four points in
any metric space, congruent to four non-coplanar points in
$\mathbb{R}^3$, can be realized on the surface of a three-dimensional
sphere using geodesic distances. The second \cite[17 May~1933,
  p.~252]{mengerK} shows that Cayley's relationship hold locally on
certain surfaces which behave locally like Euclidean spaces.

The pace quickens: in 1935, Isaac Schoenberg published some remarks on
a paper \cite{schoenberg} by Fr\'echet on the {\it Annals of
  Mathematics}, and gave, among other things, an algebraic proof of
equivalence between Euclidean Distance Matrices (EDM) and Gram
matrices. This is almost the same proof which is nowadays given to
show the validity of the classical Multidimensional Scaling (MDS)
technique \cite[\S~12.1]{borg_10}.

This brings us to the computer era, where the historical account ends
and the contemporary treatment begins. Computers allow the efficient
treatment of masses of data, some of which are incomplete and
noisy. Many of these data concern, or can be reduced to, distances,
and DG techniques are the subject of an application-oriented
renaissance \cite{dgp-sirev,dgpbook}. Motivated by the Global
Positioning System (GPS), for example, the old geographical concept of
{\it trilateration} (a system for computing the position of a point
given its distances from three known points) makes its way into DG in
wireless sensor networks \cite{eren04}. W\"uthrich's Nobel Prize for
using Nuclear Magnetic Resonance (NMR) techniques in the study of
proteins brings DG to the forefront of structural bioinformatics
research \cite{wuthrich}. The massive use of robotics in mechanical
production lines requires mathematical methods based on DG
\cite{rojasthomas}.

DG is also tightly connected with graph rigidity \cite{graverbook}.
This is an abstract mathematical formulation of statics, the study of
structures under the action of balanced forces \cite{maxwell1864b},
which is at the basis of architecture \cite{varignon}. Rigidity of
polyhedra gave rise to a conjecture of Euler's \cite{euler} about
closed polyhedral surfaces, which was proved correct only for some
polyhedra: strictly convex \cite{cauchyrigid}, convex and
higher-dimensional \cite{alexandrov3}, and generic\footnote{A
  polyhedron is generic if no algebraic relations on $\mathbb{Q}$ hold
  on the components of the vectors which represent its vertices.}
\cite{gluck}. It was however disproved in general by means of a very
special, non-generic nonconvex polyhedron \cite{connelly-countereg}.

The rest of this paper will focus on the following results, listed
here in chronological order: Heron's theorem (Sect.~\ref{s:heron}),
Euler's conjecture and Cauchy's proof for strictly convex polyhedra
(Sect.~\ref{s:eulercauchy}), Cayley-Menger determinants
(Sect.~\ref{s:cayleymenger}), Menger's axiomatization of geometry by
means of distances (Sect.~\ref{s:menger}), a result by G\"odel's
concerning DG on the sphere (Sect.~\ref{s:goedel}), and Schoenberg's
equivalence (Sect.~\ref{s:schoenberg}) between EDM and Positive
Semidefinite Matrices (PSD). There are many more results in DG: this
is simply our own choice in terms of importance and beauty.

\section{Heron's formula}
\label{s:heron}
Heron's formula, which is usually taught at school, relates the area
$\mathcal{A}$ of a triangle to the length of its sides $a,b,c$ and its
semiperimeter $s=\frac{a+b+c}{2}$ as follows:
\begin{equation}
  \mathcal{A} = \sqrt{s(s-a)(s-b)(s-c)}.
\end{equation}
There are many ways to prove its validity. Shannon Umberger, a student
of the ``Foundations of Geometry I'' course given at the University of
Georgia in the fall of 2000, proposes, as part of his final
project,\footnote{\url{http://jwilson.coe.uga.edu/emt668/emat6680.2000/umberger/MATH7200/HeronFormulaProject/finalproject.html}.}
three detailed proofs: an algebraic one, a geometric one, and a
trigonometric one. John Conway and Peter Doyle discuss Heron's formula
proofs in a publically available email
exchange\footnote{\url{https://math.dartmouth.edu/~doyle/docs/heron/heron.txt}.}
from 1997 to 2001.

Our favourite proof is based on complex numbers, and was
submitted\footnote{\url{http://www.artofproblemsolving.com/Resources/Papers/Heron.pdf}.}
to the ``Art of Problem Solving'' online school for gifted mathematics
students by Miles Edwards\footnote{Also see
  \url{http://newsinfo.iu.edu/news/page/normal/13885.html} and
  \url{http://www.jstor.org/stable/10.4169/amer.math.monthly.121.02.149}
  for more recent career achievements of this gifted student.} when he
was studying at Lassiter High School in Marietta, Georgia.

\begin{thm}[Heron's formula \cite{heron}]
Let  $\mathcal{A}$ be the area of a triangle with side lengths $a,b,c$
and semiperimeter length $s=\frac{1}{2}(a+b+c)$. Then
$\mathcal{A}=\sqrt{s(s-a)(s-b)(s-c)}$.
\label{t:heron}
\end{thm}
\begin{proof}\cite{milesedwards}
Consider a triangle with sides $a,b,c$ (opposite to the vertices $A,B,C$
respectively) and its inscribed circle centered at $O$ with radius
$r$. The perpendiculars from $O$ to the triangle sides split $a$ into
$y,z$, $b$ into $x,z$ and $c$ into $x,y$ as shown in
Fig.~\ref{f:heron1}. Let $u,v,w$ be the segments joining $O$ with
$A,B,C$, respectively.
\begin{figure}[!ht]
  \begin{center}
    \psfrag{A}{$A$}
    \psfrag{B}{$B$}
    \psfrag{C}{$C$}
    \psfrag{al}{$\alpha$}
    \psfrag{be}{$\beta$}
    \psfrag{ga}{$\gamma$}
    \psfrag{x}{$x$}
    \psfrag{y}{$y$}
    \psfrag{z}{$z$}
    \psfrag{u}{$u$}
    \psfrag{v}{$v$}
    \psfrag{w}{$w$}
    \psfrag{O}{$O$}
    \psfrag{r}{$r$}
    \psfrag{a}{$a$}
    \psfrag{b}{$b$}
    \psfrag{c}{$c$}
    \includegraphics[width=12cm]{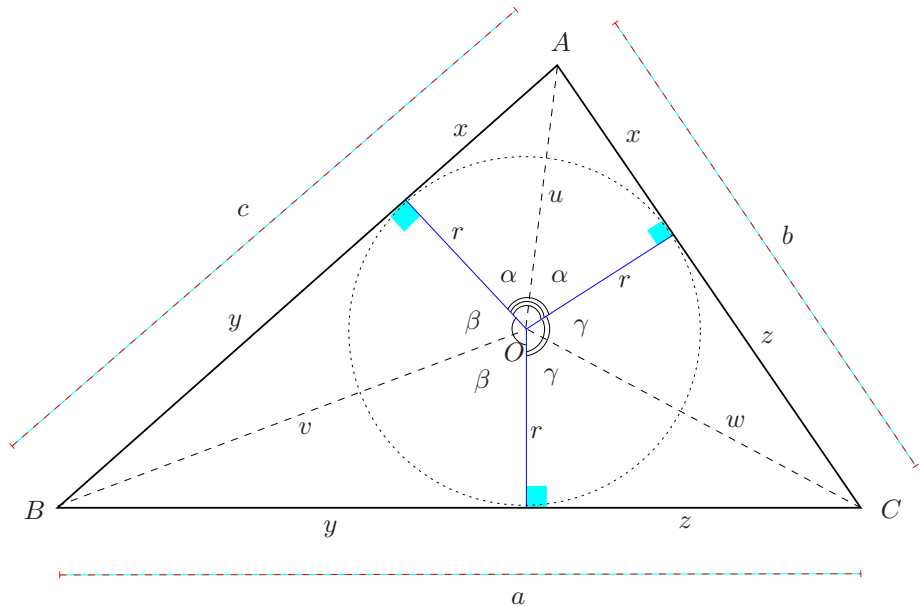}
  \end{center}
  \caption{Heron's formula: a proof using complex numbers.}
  \label{f:heron1}
\end{figure}
First, we note that $2\alpha+2\beta+2\gamma=2\pi$, which implies
$\alpha+\beta+\gamma=\pi$. Next, the following complex
identities are easy to verify geometrically in Fig.~\ref{f:heron1}:
\begin{eqnarray*}
  r + ix &=& u e^{i\alpha} \\
  r + iy &=& v e^{i\beta} \\
  r + iz &=& w e^{i\gamma}.
\end{eqnarray*}
These imply:
\begin{equation*}
  (r + ix)(r+iy)(r+iz) = (uvw) e^{i(\alpha+\beta+\gamma)} = uvw
  e^{i\pi} = -uvw,
\end{equation*}
where the last step uses Euler's identity $e^{i\pi}+1=0$ \cite[I-VIII,
  \S~138-140, p.~148]{euleraninf}. Since $-uvw$ is real, the imaginary
part of $(r+ix)(r+iy)(r+iz)$ must be zero. Expanding the product and
rearranging terms, we get $r^2(x+y+z)=xyz$. Solving for $r$, we have
the nonnegative root
\begin{equation}
r = \sqrt{\frac{xyz}{x+y+z}}. \label{eqr}
\end{equation}
We can write the semiperimeter of the triangle $ABC$ as
$s=\frac{1}{2}(a+b+c)=\frac{1}{2}(y+z+x+z+x+y)=x+y+z$. Moreover,
\begin{eqnarray*}
  s-a &=& x+y+z-y-z = x \\
  s-b &=& x+y+z-x-z = y \\
  s-c &=& x+y+z-x-y = z,
\end{eqnarray*}
so $xyz=(s-a)(s-b)(s-c)$, which implies that Eq.~\eqref{eqr} becomes:
\begin{equation*}
 r = \sqrt{\frac{(s-a)(s-b)(s-c)}{s}}.
\end{equation*}
We now write the area $\mathcal{A}$ of the triangle $ABC$ by summing it over the
areas of the three triangles $AOB$, $BOC$, $COA$, which yields:
\begin{equation*}
  \mathcal{A} =
  \frac{1}{2}(ra+rb+rc)=r\frac{a+b+c}{2}=rs=\sqrt{s(s-a)(s-b)(s-c)},
\end{equation*}
as claimed.
\end{proof}

\section{Euler's conjecture and the rigidity of polyhedra}
\label{s:eulercauchy}
Consider a square with unit sides, in the plane. One can shrink two
opposite angles and correspondingly widen the other two to obtains a
rhombus (see Fig.~\ref{f:sqrh}), which has the same side lengths but a
different shape: no sequence of rotations, translations or reflections
can turn one into the other. In other words, a square is {\it
  flexible}. By contrast, a triangle is not flexible, or {\it rigid}.
\begin{figure}[!ht]
  \begin{center}
    \includegraphics[width=8cm]{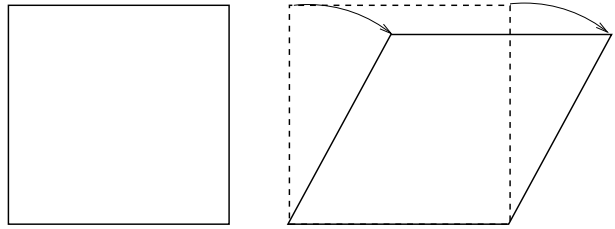}
  \end{center}
  \caption{A square is flexed into a rhombus. The set of faces (the
    edges) are the same, and each maintains pairwise distances through
    the flexing, i.e.~two points on the same edge have the same
    distance on the left as on the right figure.}
  \label{f:sqrh}
\end{figure}

Euler conjectured in 1766 \cite{euler1766} that all three-dimensional
polyhedra are rigid. The conjecture appears at the end of the
discussion about the problem {\it Invenire duas superficies, quarum
  alteram in alteram transformare liceat, ita ut in utraque singula
  puncta homologa easdem inter se teneat distantias}, i.e.:
\begin{quote}
To find two surfaces for which it is possible to transform one into
the other, in such a way that corresponding points on either keep the
same pairwise distance. ($\dag$)
\end{quote}
Towards the end of the paper, Euler writes {\it Statim enim atque
  figura undique est clausa, nullam amplius mutationem patitur}, which
means ``As soon as the shape is everywhere closed, it can no longer be
transformed''. Although the wording appears ambiguous by today's
standards, scholars of Euler and rigidity agree: what Euler really
meant is that 3D polyhedra are rigid \cite{gluck}.

To better understand this statement, we borrow from \cite{alexandrov2}
the precise definition of a {\it polyhedron}\footnote{This definition
  is different from the usual definition employed in convex analysis,
  i.e.~that a polyhedron is an intersection of half-spaces; however, a
  convex polyhedron in the sense given here is the same as a polytope
  in the sense of convex analysis.}: a family $\mathcal{K}$ of points,
open segments and open triangles is a {\it triangulation} if (a) no
two elements of $\mathcal{K}$ have common points, and (b) all sides
and vertices of the closure of any triangle of $\mathcal{K}$, and both
extreme points of the closure of any segment of $\mathcal{K}$ are all
in $\mathcal{K}$ themselves. Given a triangulation $\mathcal{K}$ in
$\mathbb{R}^K$ (where $K\in\{1,2,3\}$), the union of all points of
$\mathcal{K}$ with all points in the segments and triangles of
$\mathcal{K}$ is called a {\it polyhedron}. Note that several
triangular faces can belong to the same affine space, thereby forming
polygonal faces.

Each polyhedron has an incidence structure of points on segments and
segments on polygonal (not necessarily triangular) faces, which
induces a partial order (p.o.) based on set inclusion. For example,
the closure of the square $ABCD$ contains the closures of the segments
$AB$, $BC$, $CD$, $DA$, each of which contains the corresponding
adjacent points $A,B$, $B,C$, $C,D$, $D,A$. Accordingly, the p.o.~is
$A\subset AB,DA$; $B\subset AB,BC$; $C\subset BC,CD$; $D\subset
CD,DA$; $AB,BC,CD,DA\subset ABCD$. Since this p.o.~also has a bottom
element (the empty set) and a top element (the whole polyhedron), it
is a {\it lattice}. A lattice isomorphism is a bijective mapping
between two lattices which preserves the p.o. Two polyhedra $P,Q$ are
{\it combinatorially equivalent} if their triangulations are lattice
isomorphic. If, moreover, all the lattice isomorphic polygonal faces
of $P,Q$ are exactly equal, the polyhedra are said to be {\it facewise
  equal}.

Under the above definition, nothing prevents a polyhedron from being
nonconvex (see Fig.~\ref{f:nonconvex}). It is known that every closed
surface, independently of the convexity of its interior, is
homeomorphic (intuitively: smoothly deformable in) to some polyhedron
(again \cite[\S~2.2]{alexandrov2}). This is why we can replace
``surface'' with ``polyhedra''.
\begin{figure}[!ht]
  \begin{center}
    \includegraphics[width=5cm]{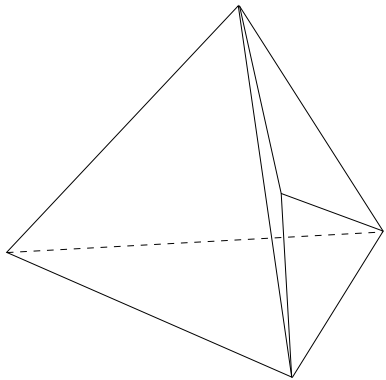}
  \end{center}
  \caption{A nonconvex polyhedron.}
  \label{f:nonconvex}
\end{figure}

The ``rigidity'' implicit in Euler's conjecture should be taken to
mean that no point of the polyhedron can undergo a continuous motion
under the constraint that the shape be the same at each point of the
motion. As for the concept of ``shape'', it is linked to that of
distance, as appears clear from ($\dag$). The following is therefore a
formal restatement of Euler's conjecture: {\it two combinatorially
  equivalent facewise equal polyhedra must be isometric under the
  Euclidean distance}, i.e.~each pair of points in one polyhedron is
equidistant with the corresponding pair in the other.

A natural question about the Euler conjecture stems from generalizing
the example in Fig.~\ref{f:sqrh} to 3D (see Fig.~\ref{f:sqrh3}). Does
this not disprove the conjecture?
\begin{figure}[!ht]
  \begin{center}
    \includegraphics[width=10cm]{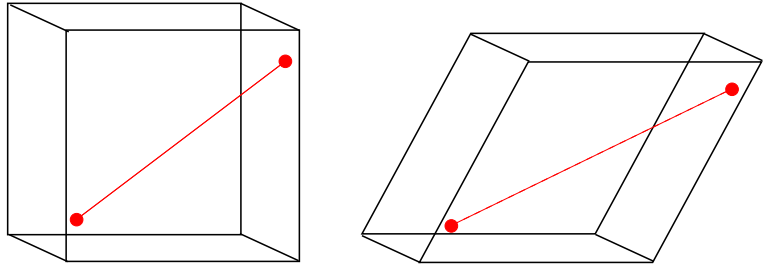}
  \end{center}
  \caption{A cube can be transformed into a rhomboid, but the set of
    faces is not the same anymore (accordingly, corresponding point
    pairs may not preserve their distance, as shown).}
  \label{f:sqrh3}
\end{figure}
The answer is no: all the polygonal faces in the cube are squares, but
this does not hold in the rhomboid. The question is more complicated
than it looks at first sight, which is why it took 211 years to
disprove it.

\subsection{Strictly convex polyhedra: Cauchy's proof}
Although Euler's conjecture is false in general, it is true for many
important subclasses of polyhedra. Cauchy proved it true for strictly
convex polyhedra.\footnote{In fact Cauchy's proof contained two
  mistakes, corrected by Steinitz \cite[p.~67]{lyusternik} and
  Lebesgue.}  There are many accounts of Cauchy's proof: Cauchy's
original text, still readable today \cite{cauchyrigid}; Alexandrov's
book \cite{alexandrov3}, Lyusternik's book \cite[\S~20]{lyusternik},
Stoker's paper \cite{stoker}, Connelly's chapter \cite{connellych}
just to name a few. Here we follow the treatment given by Pak
\cite{pak}.

We consider two combinatorially equivalent, facewise equal strictly
convex polyhedra $P,Q$, and aim to show that $P$ and $Q$ are
isometric.

For a polyhedron $P$ we consider its associated graph $G(P)=(V,E)$,
where $V$ are the points of $P$ and $E$ its segments. Note that $G(P)$
only depends on the incidence structure of the polygonal faces,
segments and points of $P$. Since $P,Q$ are combinatorially
equivalent, $G(P)=G(Q)$. Consider the dihedral angles\footnote{Two
  half-planes in $\mathbb{R}^3$ intersecting on a line $L$ define an
  angle smaller than $\pi$ called the {\it dihedral angle} at $L$.}
$\alpha_{uv},\beta_{uv}$ on $P,Q$ induced by the segment represented
by the edge $\{u,v\}\in E$. We assign to each edge $\{u,v\}\in E$ a
label $\ell_{uv}=\mbox{sgn}(\beta_{uv}-\alpha_{uv})$ (so
$\ell_{uv}\in\{-1,0,1\}$), and consider, for each $v\in V$, the edge
sequence $\sigma_v=(\{u,v\}\;|\;u\in N(v))$, where $N(v)$ is the set
of nodes adjacent to $v$. The order of the edges in $\sigma_v$ is
given by any circuit around the polygon $p(v)$ obtained by
intersecting $P$ with a plane $\gamma$ which separates $v$ from the other
vertices in $V$ (this is possible by strict convexity, see
Fig.~\ref{f:sephyp}).
\begin{figure}[!ht]
  \begin{center}
    \psfrag{v}{$v$}
    \psfrag{A}{$w_1$}
    \psfrag{B}{$w_2$}
    \psfrag{C}{$w_3$}
    \psfrag{D}{$w_4$}
    \psfrag{u1}{$u_1$}
    \psfrag{u2}{$u_2$}
    \psfrag{u3}{$u_3$}
    \psfrag{u4}{$u_4$}
    \psfrag{L}{$L$}
    \psfrag{gamma}{$\gamma$}
    \psfrag{p(v)}{$p(v)$}
    \includegraphics[width=9cm]{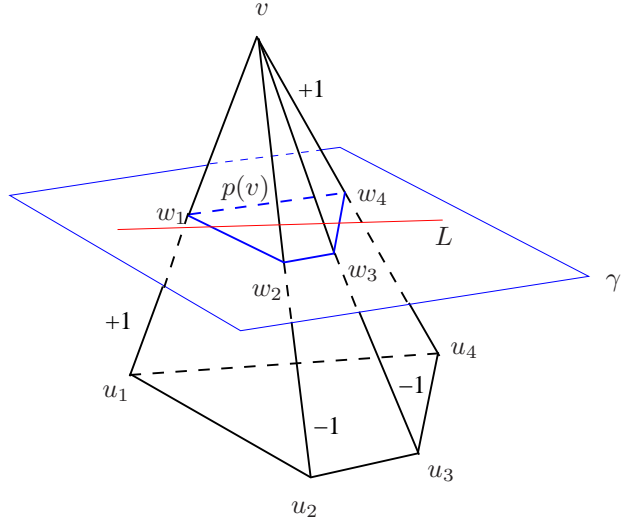}
  \end{center}
  \caption{The plane $\gamma$ separating $v$ from the other vertices in
    the strictly convex polyhedron $P$, and the intersection polygon
    $p(v)$ defined by $w_1,\ldots,w_4$. The line $L$ (lying in $p(v)$)
    separates the $+1$ and $-1$ labels applied to the points
    $w_1,\ldots,w_4$ of intersections between the edges of $P$ and
    $p(v)$.}
  \label{f:sephyp}
\end{figure}
It is easy to see that every edge $\{u,v\}\in \sigma_v$ corresponds to
a vertex of $p(v)$. Therefore, a circuit over $p(v)$ defines an order
over $\sigma_v$. We also assume that this order is periodic, i.e.~its
last element precedes the first one.  Any such sequence $\sigma_v$
naturally induces a sign sequence
$s_V=(\ell_{uv}\;|\;\{u,v\}\in\sigma_v)$; we let $\bar{s}_v$ be the
sequence $s_v$ without the zeros, and we count the number $m_v$ of
sign changes in $\bar{s}_v$, including the sign change occurring
between the last and first elements.
\begin{lem}
  For all $v\in V$, $m_v$ is even.
  \label{l:even}
\end{lem}
\begin{proof}
Suppose $m_v$ is odd, and proceed by induction on $m_v$: if $m_v=1$,
then there is only one sign change. So, the first edge $\{u,v\}$ in
$\sigma_v$ to be labelled with $\ell_{uv}\not=0$ has the property
that, going around the periodic sequence with only one sign change,
$\{u,v\}$ is also labelled with $-\ell_{uv}$, which yields $+1=-1$, a
contradiction. A trivial induction step yields the same contradiction
for all odd $m_v$.
\end{proof}
We now state a fundamental technical lemma, and provide what is
essentially Cauchy's proof, rephrased as in \cite[Lemma 2 in
  \S~20]{lyusternik}.
\begin{lem}
  If $P$ is strictly convex, then for each $v\in V$ we have either
  $m_v=0$ or $m_v\ge 4$.\label{l:signchange}
\end{lem}
\begin{proof}
By Lemma \ref{l:even}, for each $v\in V$ we have $m_v\not\in\{1,3\}$,
so we aim to show that $m_v\not=2$. Suppose, to get a contradiction,
that $m_v=2$, and consider the polygon $p(v)$ as in
Fig.~\ref{f:sephyp}. By the correspondence between edges in $\sigma_v$
and vertices of $p(v)$, the labels $\ell_{uv}$ are vertex labels in
$p(v)$.  Since there are only two sign changes, the sequence of vertex
labels can be partitioned in two contiguous sets of $+1$ and $-1$
(possibly interspersed by zeros). By convexity, there exists a line
$L$ separating the $+1$ and the $-1$ vertices (see
Fig.~\ref{f:sephyp}). Since all of the angles marked $+1$ strictly
increase, the segment $\bar{L}=L\cap p(v)$ also strictly
increases\footnote{This statement was also proved in Cauchy's paper
  \cite{cauchyrigid}, but this proof contained a serious flaw, later
  corrected by Steinitz.}; but, at the same time, all of the angles
marked $-1$ strictly decrease, so the segment $\bar{L}$ also
strictly decreases, which means that the same segment $\bar{L}$
both strictly increases and decreases, which is a contradiction (see
Fig.~\ref{s:incrdecr}).
\end{proof}
\begin{figure}[!ht]
  \begin{center}
    \psfrag{+}{\scriptsize $+1$}
    \psfrag{-}{\scriptsize $-1$}
    \psfrag{L}{\scriptsize $\bar{L}$}
    \includegraphics[width=4cm]{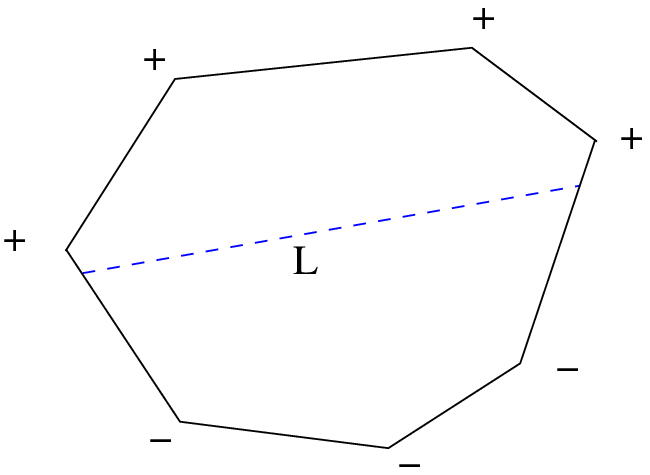} \hfill
    \psfrag{+}{\scriptsize $+1$}
    \psfrag{-}{\scriptsize $-1$}
    \psfrag{L}{\scriptsize $\bar{L}$}
    \includegraphics[width=4cm]{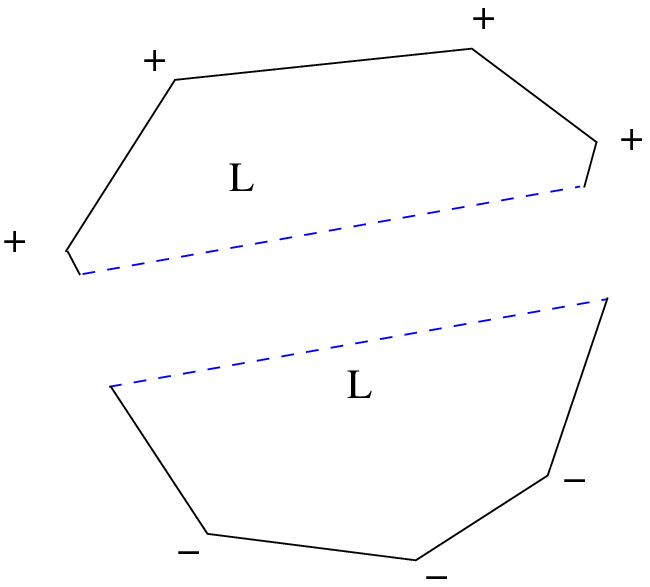} \hfill
    \psfrag{+}{\scriptsize $+1$}
    \psfrag{-}{\scriptsize $-1$}
    \psfrag{L}{\scriptsize $\bar{L}$}
    \includegraphics[width=4cm]{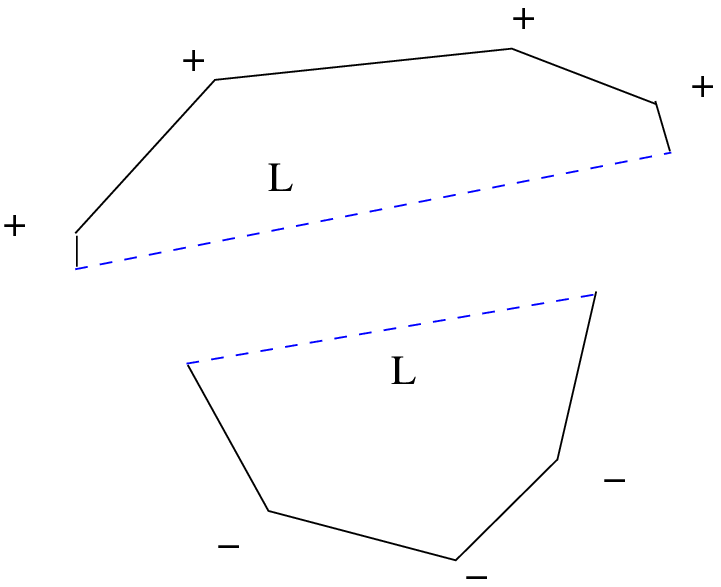}
  \end{center}
\caption{Visual representation of the contradiction in the proof of
  Lemma \ref{l:signchange}. The angles at all vertices labelled $+1$
  increase their magnitude, and those at $-1$ decrease: it follows
  that $\bar{L}$ both increases and decreases its length, a contradiction.}
\label{s:incrdecr}
\end{figure}
\begin{thm}[Cauchy's Theorem \cite{cauchyrigid}]
  If two closed convex polyhedra $P,Q$ are combinatorially equivalent
  and facewise equal, they are isometric.
\end{thm}
We only present the proof of the base case where
\begin{equation}
  \forall v\in V\;(m_v>0) \quad \vee \quad \forall v\in
  V\;(m_v=0) \label{eq:ct} 
\end{equation}
and $G(P)=G(Q)$ is a connected graph, and refer the reader to
\cite[p.~251]{pak} for the other cases (which are mostly variations of
the ideas given in the proof below).

\noindent \begin{proof} If $m_v=0$ for all $v\in V$, it means that all
  of the dihedral angles in $P$ are equals to those of $Q$, which
  implies isometry. So we assume the alternative
  w.r.t.~Eq.~\eqref{eq:ct} above: $\forall v\in V\;(m_v>0)$, and aim
  for a contradiction. Let $M=\sum_{v\in V} m_v$: by Lemma
  \ref{l:signchange} and because $m_v>0$ for each $v$, we have $M\ge
  4|V|$, a lower bound for $M$. We now construct a contradicting upper
  bound for $M$. For every $h\ge 3$, we let $F_h$ be the number of
  polygonal faces of $P$ with $h$ sides (or edges). The total number
  of polygonal faces in $P$ (or $Q$) is $\mathcal{F}=\sum_h F_h$, and
  the total number of edges is therefore
  $\mathcal{E}=\frac{1}{2}\sum_h h F_h$ (we divide by 2 since each
  edge is counted twice in the sum --- one per adjacent face --- given
  that $P,Q$ are closed). A simple term by term comparison of
  $\mathcal{F}$ and $\mathcal{E}$ yields
  $4\mathcal{E}-4\mathcal{F}=\sum_h 2(h-2)F_h$. Since each polygonal
  face $f$ of $P$ is itself closed, the number $c_f$ of sign changes
  of the quantities $\ell_{uv}$ over all edges $\{u,v\}$ adjacent to
  the face $f$ is even, by the same argument given in Lemma
  \ref{l:even}. It follows that if the number $h$ of edges adjacent to
  the face $f$ is even, then $c_f\le h$, and $c_f\le h-1$ if $h$ is
  odd. This allows us to compute an upper bound on $M$:
  \begin{eqnarray*}
    M &\le& 2 F_3 + 4F_4 +4 F_5 + 6 F_6 + 6 F_7 + 8 F_8 + \dots \\
      &\le& 2 F_3 + 4F_4 +6 F_5 + 8 F_6 + 10 F_7 + 12 F_8 + \dots \\
      &\le& 4\mathcal{E}-4\mathcal{F} = 4|V|-8.
  \end{eqnarray*}
  The middle step follows by simply increasing each coefficient. The
  last step is based on Euler's characteristic \cite{eulerchar}:
  $|V|+\mathcal{F}-\mathcal{E}=2$. Hence we have $4|V|\le M\le
  4|V|-8$, which is a contradiction.
\end{proof}

\subsection{Euler was wrong: Connelly's counterexample}
Proofs behind counterexamples can rarely be termed ``beautiful'' since
they usually lack generality (as they are applied to one particular
example). Counterexamples can nonetheless be dazzling by
themselves. Connelly's counterexample \cite{connelly-countereg} to the
Euler's conjecture consists in a very special non-generic nonconvex
polyhedron which flexes, while keeping combinatorial equivalence and
facewise equality with all polyhedra in the flex. Some years later,
Klaus Steffen produced a much simpler polyhedron with the same
properties\footnote{See \label{fn:steffen}
  \url{http://demonstrations.wolfram.com/SteffensFlexiblePolyhedron/}.}. It
is this polyhedron we exhibit in Fig.~\ref{f:steffen2}.
\begin{figure}[!ht]
  \begin{minipage}{14cm}
    \begin{center}
      \includegraphics[width=3.5cm,angle=270]{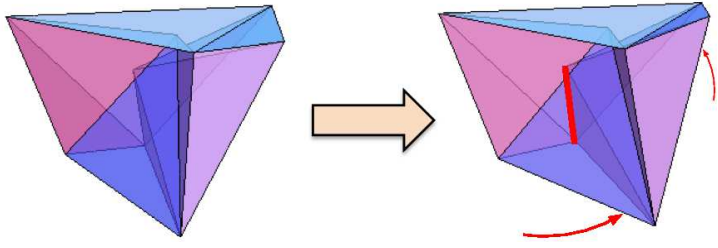}
    \end{center}
  \end{minipage}
  \caption{Steffen's polyhedron: the flex (these two images were
    obtained as snapshot from the {\it Mathematica} \cite{mathematica}
    demonstration cited in Footnote \ref{fn:steffen}). There is a
    rotation, in the direction showed by the arrows, around the edge
    which is emphasized on the right picture. The short upper right
    edge only appears shorter on the right because of perspective.}
  \label{f:steffen2}
\end{figure}

\section{Cayley-Menger determinants and the simplex volume}
\label{s:cayleymenger}
The foundation of modern DG, as investigated by Menger \cite{menger31}
and Blumenthal \cite{blumenthal}, rests on the fact that:
\begin{quote}
  {\it the four-dimensional volume of a four-dimensional simplex
    embedded in three dimensional space is zero,} ($\ast$)
\end{quote}
which we could also informally state as ``flat simplices have zero
volume''.  This is related to DG because the volume of a simplex can
be expressed in terms of the lengths of the simplex sides, which
yields a polynomial in the length of the simplex side lengths that can
be equated to zero. If these lengths are expressed in function of the
vertex positions as $\|x_u-x_v\|^2$, this yields a polynomial equation
in the positions $x_1,\ldots,x_5$ of the simplex vertices in terms of
its side lengths. Thus, if we know the positions of $x_1,\ldots,x_4$,
we can compute the unknown position of $x_5$ or prove that no such
position exists, through a process called {\it trilateration}
\cite{dvop}.

The proof of ($\ast$) was published by Arthur Cayley in 1841
\cite{cayley1841}, during his undergraduate studies. It is based on
the following well-known lemma about determinants (stated without
proof in Cayley's paper).
\begin{lem}
If $A,B$ are square matrices having the same size,
$|AB|=|A||B|$.\label{clem} 
\end{lem}
\begin{thm}[Cayley \cite{cayley1841}]
Given five points $x_1,\ldots,x_5\in\mathbb{R}^4$ all belonging to an
affine 3D subspace of $\mathbb{R}^4$, let $d_{ij}=\|x_i-x_j\|_2$ for
each $i,j\le 5$. Then
\begin{equation}
  \left|\begin{array}{cccccc}
  0 & d_{12}^2 & d_{13}^2 & d_{14}^2 & d_{15}^2 & 1 \\
  d_{21}^2 & 0 & d_{23}^2 & d_{24}^2 & d_{25}^2 & 1 \\
  d_{31}^2 & d_{32}^2 & 0 & d_{34}^2 & d_{35}^2 & 1 \\
  d_{41}^2 & d_{42}^2 & d_{43}^2 & 0 & d_{45}^2 & 1 \\
  d_{51}^2 & d_{52}^2 & d_{53}^2 & d_{54}^2 & 0 & 1 \\
  1 & 1 & 1 & 1 & 1 & 0 \end{array}\right|=0.
  \label{eq:cm0}
\end{equation}
\end{thm}

We note that Cayley's theorem is expressed for $n=5$ points in
$\mathbb{R}^3$, but it also holds for $n\ge 3$ points in
$\mathbb{R}^{n-2}$ \cite{blumenthal}. Cayley explicitly remarks that
it holds for the cases $n=4$ and $n=3$ (see \cite[VIII,
  \S~5]{sommerville} for the proof of general $n$). The determinant on
the right-hand side of Eq.~\eqref{eq:cm0} is called {\it Cayley-Menger
  determinant}, denoted by $\Delta$. We remark that in the proof below
$x_{ik}$ is the $k$-th component of $x_i$, for each $i\le 5, k\le 4$.

\noindent\begin{proof}
We follow Cayley's treatment. He pulls the following two matrices
{\small
\begin{equation*}
  A=\left(\begin{array}{cccccc}
    \|x_1\|^2 & -2x_{11} & -2x_{12} & -2x_{13} & -2x_{14} & 1 \\
    \|x_2\|^2 & -2x_{21} & -2x_{22} & -2x_{23} & -2x_{24} & 1 \\
    \|x_3\|^2 & -2x_{31} & -2x_{32} & -2x_{33} & -2x_{34} & 1 \\
    \|x_4\|^2 & -2x_{41} & -2x_{42} & -2x_{43} & -2x_{44} & 1 \\
    \|x_5\|^2 & -2x_{51} & -2x_{52} & -2x_{53} & -2x_{54} & 1 \\
    1 & 0 & 0 & 0 & 0 & 0 \end{array}\right), \quad
  B=\left(\begin{array}{cccccc}
    1 & 1 & 1 & 1 & 1 & 0 \\
    x_{11} & x_{21} & x_{31} & x_{41} & x_{51} & 0 \\
    x_{12} & x_{22} & x_{32} & x_{42} & x_{52} & 0 \\
    x_{13} & x_{23} & x_{33} & x_{43} & x_{53} & 0 \\
    x_{14} & x_{24} & x_{34} & x_{44} & x_{54} & 0 \\
    \|x_1\|^2 & \|x_2\|^2 & \|x_3\|^2 & \|x_4\|^2 & \|x_5\|^2 & 1     
  \end{array}\right)
\end{equation*}
}%
out of a magic hat. He performs the product $AB$, re-arranging and
collecting terms, and obtains a $6\times 6$ matrix where the last row
and column are $(1,1,1,1,1,0)$, and the $(i,j)$-th component is
$\|x_i-x_j\|_2^2$ for every $i,j\le 5$. To see this, it suffices to
carry out the computations using {\it Mathematica} \cite{mathematica};
by way of an example, the first diagonal component of $AB$ is
$\|x_1\|^2-2\sum\limits_{k\le 4} x_{1k}x_{1k}+\|x_1\|^2=0$, and the
component on the first row, second column of $AB$ is
$\|x_1\|^2-2\sum\limits_{k\le 4}
x_{1k}x_{2k}+\|x_2\|^2=\|x_1-x_2\|^2$. In other words, $|AB|$ is the
Cayley-Menger determinant in Eq.~\eqref{eq:cm0}. On the other hand, if
we set $x_{4k}=0$ for each $k\le 4$, effectively projecting the five
four-dimensional points in three-dimensional space, it is easy to show
that $|A|=|B|=0$ since the 5-th columns of both $A$ and the 5th row of
$B$ are zero. Hence we have $0=|A||B|=|AB|$ by Lemma \ref{clem}, and
$|AB|=0$ is precisely Eq.~\eqref{eq:cm0} as claimed.
\end{proof}

The missing link is the relationship of the Cayley-Menger determinant
with the volume of an $n$-simplex. Since this is not part of Cayley's
paper, we only establish the relationship for $n=3$. Let $d_{12}=a$,
$d_{13}=b$, $d_{23}=c$. Then:
\begin{equation*}
  \left|\begin{array}{cccccc}
  0 & a^2 & b^2  & 1 \\
  a^2 & 0 & c^2  & 1 \\
  b^2 & c^2 & 0  & 1 \\
  1 & 1 & 1 & 0 \end{array}\right| = a^4 - 2 a^2 b^2 + b^4 - 2 a^2 c^2
  - 2 b^2 c^2 + c^4 = -16(s(s-a)(s-b)(s-c)),
\end{equation*}
where $s=\frac{1}{2}(a+b+c)$ (this identity can be established by
using e.g.~{\it Mathematica} \cite{mathematica}). By Heron's theorem
(Thm.~\ref{t:heron} above) we know that the area of a triangle with
side lengths $a,b,c$ is $\sqrt{s(s-a)(s-b)(s-c)}$. So, for $n=3$, the
determinant on the left-hand side is proportional to the negative of
the square of the triangle area. This result can be generalized to
every value of $n$ \cite[II, \S~40, p.~98]{blumenthal}: it turns out
that the $n$-dimensional volume $V_n$ of an $n$-simplex in
$\mathbb{R}^n$ with side length matrix $d=(d_{ij}\;|\;i,j\le n+1)$ is:
\begin{equation*}
V_n^2 = \frac{(-1)^{n-1}}{2^n (n!)^2} \Delta.
\end{equation*}

The beauty of Cayley's proof is in its extreme compactness: it uses
determinants to hide all the details of elimination theory which would
be necessary otherwise. His paper also shows some of these details for
the simplest case $n=3$. The starting equations, as well as the
symbolic manipulation steps, depend on $n$. Although Cayley's proof is
only given for $n=5$, Cayley's treatment goes through essentially
unchanged for any number $n$ of points in dimension $n-2$.

\section{Menger's characterization of abstract metric spaces}
\label{s:menger}
At a time where mathematicians were heeding Hilbert's call to
formalization and axiomatization, Menger presented new axioms for
geometry based on the notion of distance, and provided conditions for
arbitrary sets to ``look like'' Euclidean spaces, at least
distancewise \cite{menger28,menger31}. Menger's system allows a formal
treatment of geometry based on distances as ``internal
coordinates''. The starting point is to consider the relations of
geometrical figures having proportional distances between pairs of
corresponding points, i.e.~congruence. Menger's definition of a
congruence system is defined axiomatically, and the resulting
characterization of abstract distance spaces with respect to subsets
of Euclidean spaces (possibly his most important result) transforms a
possibly infinite verification procedure (any subset of any number of
points) into a finitistic one (any subset of $n+3$ points, where $n$
is the dimension of the Euclidean space).

It is remarkable that almost none of the results below offers an
intuitive geometrical grasp, such as the proofs of Heron's formula and
Cayley's theorem do. As formal mathematics has it, part of the beauty
in Menger's work consists in turning the ``visual'' geometrical proofs
based on intuition into formal symbolic arguments based on sets and
relations. On the other hand, Menger himself gave a geometric
intuition of his results in \cite[p.~335]{mengerS}, which we comment
in Sect.~\ref{s:mengerintuitive} below.

\subsection{Menger's axioms}
Let $\mathcal{S}$ be a system of sets, and for any set
$S\in\mathcal{S}$ and any two (not necessarily distinct) points
$p,q\in S$, denote the couple $(p,q)$ by $pq$. Menger defines a
relation $\approx$ by means of the following axioms.
\begin{enumerate}
\item $\forall S,T\in\mathcal{S}$, $\forall p,q\in S$ and $\forall
  r,s\in T$, we have either $pq\approx rs$ or $pq\not\approx rs$ but
  not both.\label{M1}
\item $\forall S\in\mathcal{S}$ and $\forall p,q\in S$ we have
  $pq\approx qp$.\label{M2}
\item $\forall S,T\in\mathcal{S}$, $\forall p\in S$ and $\forall
  r,s\in T$, we have $pp\approx rs$ if and only if $r=s$.\label{M3}
\item $\forall S,T\in\mathcal{S}$, $\forall p,q\in S$ and $\forall
  r,s\in T$, if $pq\approx rs$ then $rs\approx pq$.\label{M4}
\item $\forall S,T,U\in\mathcal{S}$, $\forall p,q\in S$, $\forall
  r,s\in T$ and $\forall t,u\in U$, if $pq\approx rs$ and $pq\approx
  tu$ then $rs\approx tu$.\label{M5}
\end{enumerate}
The couple $(\mathcal{S},\approx)$ is called a {\it congruence
  system}, and the $\approx$ relation is called {\it congruence}.

Today, we are used to think of {\it relations} as defined on a single
set. We remark that in Menger's treatment, congruence is a binary
relation defined on sets of ordered pairs of points, where each point
in each pair belongs to the same set as the other, yet left-hand and
right-hand side terms may belong to different sets. We now interpret
each axiom from a more contemporary point of view.
\begin{enumerate}
  \item Axiom \ref{M1} states that Menger's congruence relation is in
    fact a {\it partial} relation on
    $\mathscr{S}=(\bigcup\mathcal{S})^2$ (the Cartesian product of the
    union of all sets $S\in\mathcal{S}$ by itself), which is only
    defined for a couple $pq\in\mathscr{S}$ whenever $\exists
    S\in\mathcal{S}$ such that $p,q\in S$.
  \item By axiom \ref{M2}, the $\approx$ relation acts on sets of {\it
    unordered} pairs of (not necessarily distinct) points; we call
    $\bar{\mathscr{S}}$ the set of all unordered pairs of points from
    all sets $S\in\mathcal{S}$.
  \item By axiom \ref{M3}, $rs$ is congruent to a pair $pq$ where
    $p=q$ if and only if $r=s$. 
  \item Axiom \ref{M4} states that $\approx$ is a symmetric relation.
  \item Axiom \ref{M5} states that $\approx$ is a transitive relation.
\end{enumerate}
Note that $\approx$ is also reflexive (i.e.~$pq\approx pq$) since
$pq\approx qp\approx pq$ by two successive applications of Axiom
\ref{M2}. So, using today's terminology, $\approx$ is an equivalence
relation defined on a subset of $\bar{\mathscr{S}}$.

\subsection{A model for the axioms}
Menger's model for his axioms is a {\it semi-metric} space $S$, i.e.~a
set $S$ of points such that to each unordered pair $\{p,q\}$ of points
in $S$ we assign a nonnegative real number $d_{pq}$ which we call {\it
  distance} between $p$ and $q$. Under this interpretation, Axiom
\ref{M2} tells us that $d_{pq}=d_{qp}$ for each pair of points $p,q$,
and Axiom \ref{M3} tells us that $rs$ is congruent to a single point
if and only if $d_{rs}=0$, which, together with nonnegativity, are the
defining properties of {\it semi-metrics} (the remaining property, the
triangular inequality, tells semi-metrics apart from {\it
  metrics}). Thus, the set $\mathcal{S}$ of all semi-metric spaces
together with the relation given by $pq\approx rs\leftrightarrow
d_{pq}=d_{rs}$ is a congruence system.

\subsection{A finitistic characterization of semi-metric spaces}
Two sets $S,T\in\mathcal{S}$ are {\it congruent} if there is a map
(called {\it congruence map}) $\phi:S\to T$, such that $pq\approx
\phi(p)\phi(q)$ for all $p,q\in S$. We denote this relation by
$S\approx_\phi T$, dropping the $\phi$ if it is clear from the
context.
\begin{lem}
  Any congruence map $\phi:S\to T$ is injective.
  \label{l:injective}
\end{lem}
\begin{proof}
Suppose, to get a contradiction, that $\exists p,q\in S$ with
$p\not=q$ and $\phi(p)=\phi(q)$: then
$pq\approx\phi(p)\phi(q)=\phi(p)\phi(p)$ and so, by Axiom \ref{M3},
$p=q$ against assumption.
\end{proof}
If $S$ is congruent to a subset of $T$, then we say that $S$ is {\it
  congruently embeddable} in $T$.

\subsubsection{Congruence order}
Now consider a set $S\in\mathcal{S}$ and an integer $n\ge 0$
with the following property: for any $T\in\mathcal{S}$, if all
$n$-point subsets of $T$ are congruent to an $n$-point subset of $S$,
then $T$ is congruently embeddable in $S$. If this property holds,
then $S$ is said to have {\it congruence order} $n$.  Formally, the
property is written as follows:
\begin{equation}
  \forall T\in\mathcal{S} \; \forall T'\subseteq T \quad (\ (|T'|=n\to
  \exists S'\subseteq S\;(|S'|=n\land T'\approx S')) \quad
  \longrightarrow \quad \exists R\subseteq S\;(T\approx
  R)\ ). \label{eq:congruenceorder}
\end{equation}

If $|S|<n$ for some positive integer $n$, then $S$ can have congruence
order $n$, since the definition is vacuously satisfied. So we assume
in the following that $|S|\ge n$.
\begin{prop}
If $S$ has congruence order $n$ in $\mathcal{S}$, then it also has
congruence order $m$ for each $m>n$. \label{p:comn}
\end{prop}
\begin{proof}
  By hypothesis, for every $T\in\mathcal{S}$, if every $n$-point
  subset $T'$ of $T$ is congruent to an $n$-point subset of $S$, then
  there is a subset $R$ of $S$ such that $T\approx_\phi R$. Now any
  $m$-point subset of $S$ is mapped by $\phi$ to a congruent $m$-point
  subset of $S$, and again $T\approx R$, so
  Eq.~\eqref{eq:congruenceorder} is satisfied for $S$ and $m$.
\end{proof}

\begin{prop}
$\mathbb{R}^0$ (i.e.~the Euclidean space which simply consists of the
  origin) has minimum congruence order $2$ in $\mathcal{S}$.
\end{prop}
\begin{proof}
  Pick any $T\in\mathcal{S}$ with $|T|>1$. None of its $2$-point
  subsets is congruent to any $2$-point subset of $\mathbb{R}^0$,
  since none exists. Moreover, $T$ itself cannot be congruently
  embedded in $\mathbb{R}^0$, since $|T|>1=|\mathbb{R}^0|$ and no
  injective congruence map can be defined, against Lemma
  \ref{l:injective}. So the integer $2$ certainly (vacuously)
  satisfies Eq.~\eqref{eq:congruenceorder} for $S=\mathbb{R}^0$, which
  means that $\mathbb{R}^0$ has congruence order $2$. In view of
  Prop.~\ref{p:comn}, it also has congruence order $m$ for each
  $m>2$. Hence we have to show next that the integer $1$ cannot be a
  congruence order for $\mathbb{R}^0$. To reach a contradiction,
  suppose the contrart, and let $T$ be as above. By Axiom \ref{M3},
  every singleton subset of $T$ is congruent to a subset of
  $\mathbb{R}^0$, namely the subset containing the origin. Thus, by
  Eq.~\eqref{eq:congruenceorder}, $T$ must be congruent to a subset of
  $\mathbb{R}^0$; but, again, $|T|>1=|\mathbb{R}^0|$ contradicts Lemma
  \ref{l:injective}: so $T$ cannot be congruently embedded in
  $\mathbb{R}^0$, which negates Eq.~\eqref{eq:congruenceorder}. Hence
  $1$ cannot be a congruence order for $\mathbb{R}^0$, as claimed.
\end{proof}

\subsubsection{Menger's fundamental result}
The fundamental result proved by Menger in 1928 \cite{menger28} is
that the Euclidean space $\mathbb{R}^n$ has congruence order $n+3$ but
not $n+2$ for each $n>0$ in the family $\mathcal{S}$ of all
semi-metric spaces. The important implication of Menger's result is
that in order to verify whether an abstract semi-metric space is
congruent to a subset of a Euclidean space, we only need to verify
congruence of each of its $n+3$ point subsets.

We follow Blumenthal's treatment \cite{blumenthal}, based on the
following preliminary definitions and properties, which we shall not prove:
\begin{enumerate}
  \item A congruent mapping of a semi-metric space onto itself is
    called a {\it motion}; \label{BD1}
  \item $n+1$ points in $\mathbb{R}^{n}$ are {\it independent} if they
    are not affinely dependent (i.e.~if they do not all belong to a
    single hyperplane in $\mathbb{R}^n$); \label{BD2}
  \item two congruent $(n+1)$-point subsets of $\mathbb{R}^n$ are either
    both independent or both dependent; \label{BP1}
  \item there is at most one point of $\mathbb{R}^n$ with given
    distances from an independent $(n+1)$-point subset;\label{BP2}
  \item any congruence between any two subsets of $\mathbb{R}^n$ can
    be extended to a motion;\label{BP3}
  \item any congruence between any two independent $(n+1)$-point
    subsets of $\mathbb{R}^n$ can be extended to a unique
    motion.\label{BP4}
\end{enumerate}
\begin{thm}[Menger \cite{menger28}] A non-empty semi-metric space $S$ is
  congruently embeddable in $\mathbb{R}^n$ (but not in any
  $\mathbb{R}^r$ for $r<n$) if and only if: (a) $S$ contains an
  $(n+1)$-point subset $S'$ which is congruent with an independent
  $(n+1)$-point subset of $\mathbb{R}^n$; and (b) each $(n+3)$-point
  subset $U$ of $S$ containing $S'$ is congruent to an $(n+3)$-point
  subset of $\mathbb{R}^n$.
\end{thm}
The proof of Menger's theorem is very formal (see below) and somewhat
difficult to follow. It is nonetheless a good example of a proof in an
axiomatic setting, where logical reasoning is based on syntactical
transformations induced by inference rules on the given axioms. An
intuitive discussion is provided in Sect.~\ref{s:mengerintuitive}.

\noindent\begin{proof}
($\Rightarrow$) Assume first that $S\approx_\phi
  T\subseteq\mathbb{R}^n$, where the affine closure of $T$ has
  dimension $n$. Then $T$ must contain an independent subset $T'$ with
  $|T'|=n+1$, which we can map back to a subset $S'\subseteq S$ using
  $\phi^{-1}$. Since $\phi,\phi^{-1}$ are injective, $|S'|\le |T'|$,
  and by Axiom \ref{M3} we have $|S'|\ge |T'|$, so $|S'|=n+1$, which
  establishes (a). Now take any $U\subseteq S$ with $|U|=n+3$ and
  $U\supset S'$: this can be mapped via $\phi$ to a subset $W\subseteq
  T$: Lemma \ref{l:injective} ensures injectivity of $\phi$ and hence
  $|W|=n+3$, establishing (b).\\ ($\Leftarrow$) Conversely, assume (a)
  and (b) hold. By (a), let $S'\subseteq S$ with $|S'|=n+1$ and
  $S'\approx_\phi T'\subseteq\mathbb{R}^n$, with $T'$ independent and
  $|T'|=n+1$. We claim that $\phi$ can be extended to a mapping of $S$
  into $\mathbb{R}^n$. Take any $q\in S\smallsetminus S'$: by (b),
  $S'\cup\{q\}\approx_\psi W\subseteq \mathbb{R}^n$ with
  $|W|=n+2$. Note that $T'\approx_\omega W\smallsetminus\{\psi(q)\}$
  by Axiom \ref{M5}, which implies that for any $p\in S'$, we have
  $\omega\phi(p)=\psi(p)$.  Moreover, by Property \ref{BP1} above,
  $W\smallsetminus\{\psi(q)\}$ is independent and has cardinality
  $n+1$, which by Property \ref{BP4} above implies that $\omega$ can
  be extended to a unique motion in $\mathbb{R}^n$. So the action of
  $\omega$ is extended to $q$, and we can define
  $\phi(q)=\omega^{-1}\psi(q)$. We now show that this extension of
  $\phi$ is a congruence. Let $p,q\in S$: we aim to prove that
  $pq=\phi(p)\phi(q)$. Consider the set $U=S'\cup\{p,q\}$: since
  $|U|\le n+3$, by (b) there is $W\subset\mathbb{R}^n$ with $|W|=|U|$
  such that $U\approx_\psi W$. As above, we note that there is a
  subset $W'\subseteq W$ such that $|W'|=n+1$ and $T'\approx_\omega
  W'$, that $\omega\phi(r)=\psi(r)$ for each $r\in S'$, and that
  $\omega$ is a motion of $\mathbb{R}^n$. Hence
  $pq=\psi(p)\psi(q)=\omega^{-1}\phi(p)\omega^{-1}\phi(q)=\phi(p)\phi(q)$,
  as claimed.
\end{proof}

\subsection{An intuitive interpretation}
\label{s:mengerintuitive}
Although we stated initially that part of the the beauty of the formal
treatment of geometry is that it is based on symbolic manipulation
rather than visual intuition, we quote from a survey paper which
Menger himself wrote (in Italian, with the help of L.~Geymonat) to
disseminate the work carried out at his seminar \cite{mengerS}.
\begin{quote}
{\it Affinch\'e uno spazio metrico reale $R$ sia applicabile a un
  insieme parziale di $\mathbb{R}^n$ \`e necessario e sufficiente che
  per ogni $n+3$ e per ogni $n+2$ punti di esso sia $\Delta=0$ e
  inoltre che ogni $n+1$ punti di $R$ siano applicabili a punti di
  $\mathbb{R}^n$.} 
\end{quote}
The translation is ``a real metric space $R$ is embeddable in a subset
of $\mathbb{R}^n$ if and only if $\Delta=0$ for each $(n+3)$- and
$(n+2)$-point subsets or $R$, and that each $(n+1)$-point subset of
$R$ is embeddable in $\mathbb{R}^n$.''

Since we know that $\Delta$, the Cayley-Menger determinant of the
pairwise distances of a set $S$ of points (see Eq.~\eqref{eq:cm0}), is
proportional to the volume of the simplex on $S$ embedded in $|S|-1$
dimensions, what Menger is saying is that his result on the congruence
order of Euclidean spaces can be intuitively interpreted as follows.
\begin{quote}
An abstract semi-metric space $R$ is congruently embeddable in
$\mathbb{R}^n$ if and only if: (i) there are $n+1$ points in $R$ which
are congruently embeddable in $\mathbb{R}^n$; (ii) the volume of the
simplex on each $n+2$ points of $R$ is zero; (iii) the volume of the
simplex on each $n+3$ points of $R$ is zero.
\end{quote}
This result is exploited in the algorithm for computing point
positions from distances given in \cite[p.~2284]{sippl}.

\section{G\"odel on spherical distances}
\label{s:goedel}
Kurt G\"odel's name is attached to what is possibly the most
revolutionary result in all of mathematics, i.e.~G\"odel's
incompleteness theorem, according to which any formal axiomatic system
sufficient to encode the integers is either inconsistent (it proves
$A$ and $\neg A$) or incomplete (there is some true statement $A$
which the system cannot prove). This shattered Hilbert's dream of a
formal system in which every true mathematical statement could be
proved. Few people know that G\"odel, who attended the Vienna Circle,
Menger's course in geometry, and Menger's seminar, also contributed
two results which are completely outside of the domain of logic. These
results only appeared in the proceedings of Menger's seminar
\cite{mengerK}, and concern DG on a spherical surface.

\subsection{Four points on the surface of a sphere}
The result we discuss here is a proof to the following theorem,
conjectured at a previous seminar session by Laura Klanfer. We remark
that a sphere in $\mathbb{R}^3$ is a semi-metric space whenever it is
endowed with a distance corresponding to the length of a geodesic
curve joining two points.
\begin{thm}[G\"odel \cite{goedelDG1}]
  Given a semi-metric space $S$ of four points, congruently embeddable
  in $\mathbb{R}^3$ but not $\mathbb{R}^2$, is also congruently
  embeddable on the surface of a sphere in
  $\mathbb{R}^3$. \label{t:goedel}
\end{thm}
G\"odel's proof looks at the circumscribed sphere around a tetrahedron
in $\mathbb{R}^3$, and analyses the relationship of the geodesics,
their corresponding chords, and the sphere radius. It then uses a
fixed point argument to find the radius which corresponds to geodesics
which are as long as the given sides.

\noindent\begin{proof} The congruence embedding of $S$ in
$\mathbb{R}^3$ defines a tetrahedron $T$ having six (straight) sides
with lengths $a_1,\ldots,a_6$.  Let $r$ be the radius of the sphere
circumscribed around $T$ (i.e.~the smallest sphere containing $T$). We
shall now consider a family of tetrahedra $\tau(x)$, parametrized on a
scalar $x>0$, defined as follows: $\tau(x)$ is the tetrahedron in
$\mathbb{R}^3$ having side lengths $c_x(a_1),\ldots,c_x(a_6)$, where
$c_x(\alpha)$ is the length of the chord subtending a geodesic having
length $\alpha$ on a sphere of radius $\frac{1}{x}$. As $x$ tends
towards zero, each $c_x(a_i)$ tends towards $a_i$ (for each $i\le 6$),
since the radius of the sphere tends towards infinity and each
geodesic length tends towards the length of the subtending chord. This
means that $\tau(x)$ tends towards $T$, since $T$ is precisely the
tetrahedron having side lengths $a_1,\ldots,a_6$. For each $x>0$, let
$\phi(x)$ be the inverse of the radius of the sphere circumscribed
about $\tau(x)$. Since $\tau(x)\to T$ as $x\to 0$, and the radius
circumscribed about $T$ is $r$, it follows that
$\phi(x)\to\frac{1}{r}$ as $x\to 0$. Also, since $T$ exists by
hypothesis, we can define $\tau(0)=T$ and $\phi(0)=\frac{1}{r}$. Also
note that it is well known by elementary spherical geometry that:
\begin{equation}
  c_x(\alpha) = \frac{2}{x}\sin\frac{\alpha x}{2}. \label{eq:cx}
\end{equation}
\underline{Claim}: if $a'=\max\{a_1,\ldots,a_6\}$ then $\phi$ has a
fixed point in the open interval
$I=(0,\frac{\pi}{a'})$. \\ \underline{\it Proof of the claim}. First
of all notice that $\tau(0)$ exists, and $c_x(\alpha)$ is a continuous
function for $x>0$ for each $\alpha$ (by Eq.~\eqref{eq:cx}). Since
$\tau(x)$ is defined by the chord lengths $c_x(a_1),\ldots,c_x(a_6)$,
this also means that $\tau(x)$ varies continuously for $x$ in some
open interval $J=(0,\varepsilon)$ (for some constant $\varepsilon>0$).
In turn, this implies that $\bar{x}=\max \{y\in I \;|\;\tau(y)\mbox{
  exists}\}$ exists by continuity. There are two cases: either
$\bar{x}$ is at the upper extremum of $I$, or it is not.
\begin{itemize}
  \setlength{\parskip}{-0.2em}
  \item[(i)] If $\bar{x}=\frac{\pi}{a'}$, then $\tau(\bar{x})$ exists,
    its longest edge has length $c_{\bar{x}}(a')=\frac{2a'}{\pi}$, so,
    by elementary spherical geometry, the radius of the sphere
    circumscribed around $\tau(\bar{x})$ is greater than
    $\frac{c_{\bar{x}}(a')}{2}$, i.e.~greater than
    $\frac{a'}{\pi}=\frac{1}{\bar{x}}$. Thus
    $\phi(\bar{x})<\bar{x}$. We also have, however, that
    $\phi(0)=\frac{1}{r}>0$, so by the intermediate value theorem
    there must be some $x\in(0,\bar{x})$ with $\phi(x)=x$.
  \item[(ii)] Assume now $\bar{x}<\frac{\pi}{a'}$ and suppose
    $\tau(\bar{x})$ is non-planar. Then for each $y$ in an arbitrary
    small neighbourhood around $\bar{x}$, $\tau(y)$ must exist by
    continuity: in particular, there must be some $y>\bar{x}$ where
    $\tau(y)$ exists, which contradicts the definition of
    $\bar{x}$. So $\tau(\bar{x})$ is planar: this means that each
    geodesic is contained in the same plane, which implies that the
    geodesics are linear segments. It follows that the circumscribed
    sphere has infinite radius, or, equivalently, that
    $\phi(\bar{x})=0<\bar{x}$. Again, by $\phi(0)>0$ and the
    intermediate value theorem, there must be some $x\in(0,\bar{x})$
    with $\phi(x)=x$.
\end{itemize}
This concludes the proof of the claim. \\ So now let $y$ be the fixed
point of $\phi$. The tetrahedron $\tau(y)$ has side lengths $c_y(a_i)$
for each $i\le 6$, and is circumscribed by a sphere $\sigma$ with
radius $\frac{1}{y}$. It follows that, on the sphere $\sigma$, the
geodesics corresponding to the chords given by the tetrahedron sides
have lengths $a_i$ (for $i\le 6$), as claimed.
\end{proof}

\subsection{G\"odel's devilish genius}
\label{s:genius}
G\"odel's proof exhibits an unusual peak of devilish genius. At first
sight, it is a one-dimensional fixed-point argument which employs a
couple of elementary notions in spherical geometry. Underneath the
surface, the fixed-point argument eschews a misleading visual
intuition.

$T$ is a given tetrahedron in $\mathbb{R}^3$ which is assumed to be
non-planar and circumscribed by a sphere of finite positive radius
$r$ (see Fig.~\ref{f:goedel}, left).
\begin{figure}[!ht]
  \begin{center}
    \psfrag{a1}{$a_1$}
    \psfrag{a2}{$a_2$}
    \psfrag{a3}{$a_3$}
    \psfrag{a4}{$a_4$}
    \psfrag{a5}{$a_5$}
    \psfrag{a6}{$a_6$}
    \psfrag{T}{$T$}
    \psfrag{T(x)}{$\tau(x)$}
    \psfrag{1/x}{$\frac{1}{\phi(x)}$}
    \includegraphics[width=14cm]{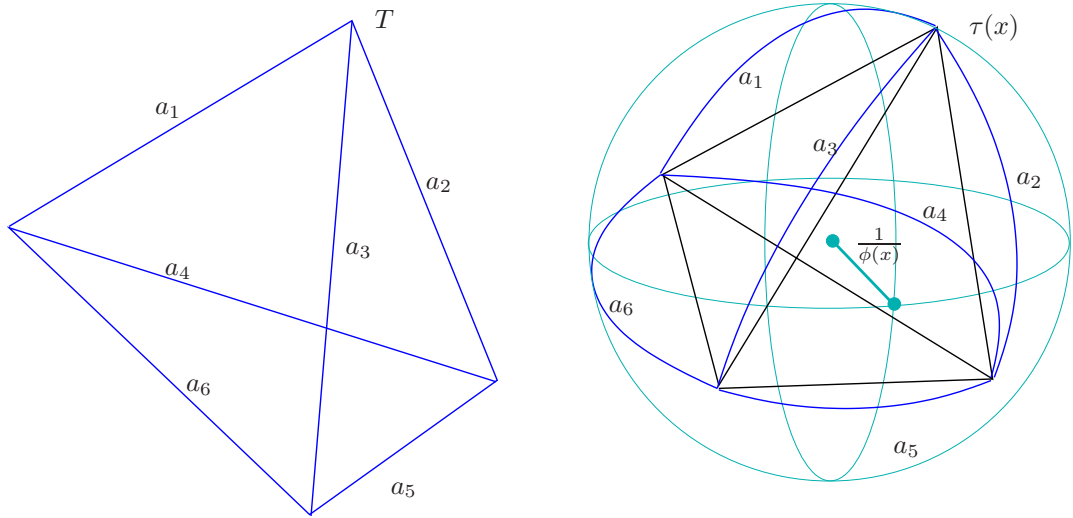}
  \end{center}
  \caption{The given tetrahedron $T$ (left), and the tetrahedron
    $\tau(x)$ (right). Beware of this visual interpretation: it may
    yield misleading insights (see Sect.~\ref{s:genius}).}
  \label{f:goedel}
\end{figure}
The map $\tau$ sends a scalar $x$ to the tetrahedron having as side
lengths the chords subtending the geodesics of length $a_i$ ($i\le 6$)
on a sphere of radius $\frac{1}{x}$ (see Fig.~\ref{f:goedel},
right). The map $\tau$ is such that $\tau(0)=T$ since for $x=0$ the
radius is infinite, which means that the geodesics are equal to their
chords. Moreover, the map $\phi$ sends $x$ to the inverse of the
radius of the sphere circumscribing $\tau(x)$. Since every geodesic on
the sphere is a portion of a great circle, it would appear from
Fig.~\ref{f:goedel} (right) that the radius $\frac{1}{x}$ used to
compute $c_x(a_i)$ ($i\le 6$) is the same as the radius
$\frac{1}{\phi(x)}$ of the sphere circumscribing $\tau(x)$, which
would immediately yield $\phi(x)=x$ for every $x$ --- making the proof
trivial. There is something inconsistent, however, in the visual
interpretation of Fig.~\ref{f:goedel}: the given tetrahedron $T$
corresponds to the case $\tau(x)=T$, which happens when $x=0$,
i.e.~the radius of the sphere circumscribed around $T$ is
$\infty$. But this would yield $T$ to be a planar tetrahedron, which
is a contradiction with an assumption of the theorem. Moreover, if
$\phi(x)$ were equal to $x$ for each $x$, this would yield
$0=\phi(0)=\frac{1}{r}>0$, another contradiction.


The misleading concept is hidden in the picture in Fig.~\ref{f:goedel}
(right). It shows a tetrahedron inscribed in a sphere, and a spherical
tetrahedron {\it on the same vertices}. This is not true in general,
i.e.~the spherical tetrahedron with the given curved side lengths
$a_1,\ldots,a_6$ cannot, in general, be embedded in the surface of a
sphere of {\it any} radius. For example, the case $x=0$ yields
geodesics with infinite curvatures (i.e.~straight lines laying in a
plane), but $\phi(x)=\frac{1}{r}>0$, and there is no flat tetrahedron
with the same distances as those of $T$. The sense of G\"odel's proof
is that the function $c_x$ simply transforms a set of geodesic
distances into a set of linear distances, i.e.~it maps scalars to
scalars rather than geodesics to segments, whereas Fig.~\ref{f:goedel}
(right) shows the special case where the geodesics are mapped to the
corresponding segments, with intersections at the same points (namely
the distances $a_1,\ldots,a_6$ can be embedded on the particular
sphere shown in the picture). More specifically, the geodesic curves
may or may not be realizable on a sphere of radius
$\frac{1}{\phi(x)}$. G\"odel's proof shows exactly that there must be
some $x$ for which $\phi(x)=x$, i.e.~the geodesic curves become
realizable.

\subsection{Existential vs.~constructive proofs}
Like many existential proofs based on fixed-point
theorems,\footnote{Interestingly, G\"odel's famous incompleteness
  theorem is also a fixed-point argument (in a much more complicated
  set).} this proof is beautiful because it asserts the truth of the
theorem without any certificates other than its own logical
validity. An alternative, constructive proof of Thm.~\ref{t:goedel} is
given in \cite[Thm.~3']{schoenberg}. The tools used in that proof,
Cayley-Menger determinants and positive semidefiniteness, are
discussed in Sect.~\ref{s:schoenberg} below.

\section{The equivalence of EDM and PSD matrices}
\label{s:schoenberg}
Many fundamental innovations stem from what are essentially footnotes
to apparently deeper or more important work. Isaac Schoenberg, better
known as the inventor of splines \cite{splines}, published a paper in
1935 titled {\it Remarks to Maurice Fr\'echet's article ``Sur la
  d\'efinition axiomatique d'une classe d'espace distanci\'es
  vectoriellement applicable sur l'espace de Hilbert''}
\cite{schoenberg}. The impact of Schoenberg's remarks far exceeds that
of the original paper\footnote{A not altogether dissimilar situation
  arose for the Johnson-Lindenstrauss (JL) lemma \cite{jllemma}: the
  paper is concerned with extending a mapping from $n$-point subsets
  of a metric space to the whole metric space in such a way that the
  Lipschitz constant of the extension is bounded by at most a constant
  factor. Johnson and Lindenstrauss state on page 1 that ``The main
  tool for proving Theorem 1 is a simply stated elementary geometric
  lemma''. This lemma is now known as the {\it JL lemma}, and
  postulates the existence of low-distortion projection matrices which
  map to Euclidean spaces of logarithmically fewer dimensions. The
  impact of the lemma far exceeds that of the main result.}: these
remarks encode what amounts to the basis of the well-known MDS
techniques for visualizing high-dimensional data \cite{coxcox}, as
well as all the solution techniques for Distance Geometry Problems
(DGP) based on Semidefinite Programming (SDP) \cite{ye,wolkowicz}.

\subsection{Schoenberg's problem}
Schoenberg poses the following problem, relevant to Menger's treatment
of distance geometry \cite[p.~737]{menger31}.
\begin{quote}
  Given an $n\times n$ symmetric matrix $D$, what are necessary and
  sufficient conditions such that $D$ is a EDM corresponding to $n$
  points in $\mathbb{R}^r$, with $1\le r\le n$ minimum?
\end{quote}
Menger's solution is based on Cayley-Menger determinants; Schoenberg's
solution is much simpler and more elegant, and rests upon the
following theorem. Recall that a matrix is PSD if and only if all its
eigenvalues are nonnegative.
\begin{thm}[Schoenberg \cite{schoenberg}]
The $n\times n$ symmetric matrix $D=(d_{ij})$ is the EDM of a set of
$n$ points $x=\{x_1,\ldots,x_n\}\subset\mathbb{R}^{r}$ (with $r$
minimum) if and only if the matrix
$G=\frac{1}{2}(d_{1i}^2+d_{1j}^2-d_{ij}^2\;|\;2\le i,j\le n)$ is PSD
of rank $r$.\label{t:schoenberg}
\end{thm}

Instead of providing Schoenberg's proof, we follow a more modern
treatment, which also unearths the important link of this theorem with
classical MDS \cite[\S~2.2.1]{coxcox}, an approximate method for
finding sets of points $x=\{x_1,\ldots,x_n\}$ having EDM
which approximates a given symmetric matrix. MDS is one of the
cornerstones of the modern science of data analysis.

\subsection{The proof of Schoenberg's theorem}
\label{s:schthm}
Given a set $x=\{x_1,\ldots,x_n\}$ of points in $\mathbb{R}^r$, we can
write $x$ as an $r\times n$ matrix having $x_i$ as $i$-th column. The
matrix $G=x\transpose{x}$ having the scalar product $x_ix_j$ as its
$(i,j)$-th component is called the {\it Gram matrix} or {\it Gramian}
of $x$. The proof of Thm.~\ref{t:schoenberg} works by exhibiting a 1-1
correspondence between squared EDMs and Gram matrices, and then by
proving that a matrix is Gram if and only if it is PSD.

Without loss of generality, we can assume that the barycenter of the
points in $x$ is at the origin:
\begin{equation}
  \sum_{i\le n} x_i = 0. \label{eq:mds0}
\end{equation}
Now we remark that, for each $i,j\le n$, we have:
\begin{equation}
  d^2_{ij} =
  \|x_i-x_j\|^2=(x_i-x_j)(x_i-x_j)=x_ix_i+x_jx_j-2x_ix_j. \label{eq:mds1}
\end{equation}

\subsubsection{The Gram matrix in function of the EDM}
We ``invert'' Eq.~\eqref{eq:mds1} to compute the matrix
$G=x\transpose{x}=(x_i x_j)$ in function of the matrix
$D^2=(d_{ij}^2)$. We sum Eq.~\eqref{eq:mds1} over all values of
$i\in\{1,\ldots,n\}$, obtaining:
\begin{equation}
  \sum_{i\le n} d^2_{ij} = \sum_{i\le n} (x_i x_i) + n (x_j x_j)
  - 2\left(\sum_{i\le n} x_i\right)  x_j. \label{eq:mds2}
\end{equation}
By Eq.~\eqref{eq:mds0}, the negative term in the right hand side of
Eq.~\eqref{eq:mds2} is zero. On dividing through by $n$, we have
\begin{equation}
  \frac{1}{n} \sum_{i\le n} d^2_{ij} = \frac{1}{n} \sum_{i\le n}
  (x_i x_i) + x_j x_j. \label{eq:mds3}
\end{equation}
Similarly for $j\in\{1,\ldots,n\}$, we obtain:
\begin{equation}
  \frac{1}{n} \sum_{j\le n} d^2_{ij} = x_i x_i + \frac{1}{n} \sum_{j\le n}
  (x_j x_j). \label{eq:mds4}
\end{equation}
We now sum Eq.~\eqref{eq:mds3} over all $j$, getting:
\begin{equation}
  \frac{1}{n} \sum_{i\le n\atop j\le n} d^2_{ij} =
  n\frac{1}{n}\sum_{i\le n} (x_i x_i) + \sum_{j\le n} (x_j
  x_j) = 2\sum_{i\le n} (x_i x_i) \label{eq:mds5}
\end{equation}
(the last equality in Eq.~\eqref{eq:mds5} holds because the same
quantity $f(k)=x_k x_k$ is being summed over the same range
$\{1,\ldots,n\}$, with the symbol $k$ replaced by the symbol $i$ first
and $j$ next). We then divide through by $n$ to get:
\begin{equation}
  \frac{1}{n^2} \sum_{i\le n\atop j\le n} d^2_{ij} = \frac{2}{n}
  \sum_{i\le n} (x_i x_i). \label{eq:mds6}
\end{equation}

We now rearrange Eq.~\eqref{eq:mds1}, \eqref{eq:mds4},
\eqref{eq:mds3} as follows:
\begin{eqnarray}
  2 x_i x_j &=& x_i x_i + x_j x_j - d_{ij}^2 
  \label{eq:mds7} \\
  x_i x_i &=& \frac{1}{n} \sum_{j\le n} d^2_{ij} - \frac{1}{n} \sum_{j\le n}
  (x_j x_j) \label{eq:mds8} \\
  x_j x_j &=& \frac{1}{n} \sum_{i\le n} d^2_{ij} - \frac{1}{n} \sum_{i\le n}
  (x_i x_i), \label{eq:mds9} 
\end{eqnarray}
and replace the left hand side terms of
Eq.~\eqref{eq:mds8}-\eqref{eq:mds9} into Eq.~\eqref{eq:mds7} to
obtain:
\begin{equation}
  2 x_i x_j = \frac{1}{n} \sum_{k\le n} d^2_{ik} +
  \frac{1}{n} \sum_{k\le n} d^2_{kj} - d_{ij}^2 - \frac{2}{n} \sum_{k\le n}
  (x_k x_k),  
  \label{eq:mds10}   
\end{equation}
whence, on substituting the last term using Eq.~\eqref{eq:mds6}, we
have:
\begin{equation}
  2 x_i x_j = \frac{1}{n} \sum_{k\le n} (d^2_{ik} + d^2_{kj}) -
  d_{ij}^2 - \frac{1}{n^2} \sum_{h\le n\atop k\le n} d^2_{hk}.
  \label{eq:mds11}   
\end{equation}

It turns out that Eq.~\eqref{eq:mds11} can be written in matrix form
as:
\begin{equation}
   G = -\frac{1}{2} J D^2 J, \label{eq:mds12}
\end{equation}
where $J=I_n-\frac{1}{n}\mathbf{1}\transpose{\mathbf{1}}$ and
$\mathbf{1}=\underbrace{(1,\ldots,1)}_n$. 

\subsubsection{Gram matrices are PSD matrices}
Any Gram matrix $G=x\transpose{x}$ derived by a point sequence (also
called a {\it realization}) $x=(x_1,\ldots,x_n)$ in $\mathbb{R}^K$ for
some non-negative integer $K$ has two important properties: (i) the
rank of $G$ is equal to the rank of $x$; and (ii) $G$ is PSD,
i.e.~$\transpose{y} G y\ge 0$ for all $y\in\mathbb{R}^n$. For
simplicity, we only prove these properties in the case when
$x=(x_1,\ldots,x_n)$ is a $1\times n$ matrix, i.e.~$x\in\mathbb{R}^n$,
and $x_i$ is a scalar for all $i\le n$ (this is the case $r=1$ in
Schoenberg's problem above).
\begin{itemize}
\item[(i)] The $i$-th column of $G$ is the vector $x$ multiplied by
  the scalar $x_i$, which means that every column of $G$ is a scalar
  multiple of a single column vector, and hence that $\mbox{\sf
    rk}\,G=1$;
\item[(ii)] For any vector $y$, $\transpose{y} G
  y=\transpose{y}(x\transpose{x}) y = (\transpose{y} x)(\transpose{x}
  y)=(\transpose{x} y)^2\ge 0$.
\end{itemize}
Moreover, $G$ is a Gram matrix {\it only if} it is PSD. Let $M$ be a
PSD matrix. By spectral decomposition there is a unitary matrix $Y$
such that $M=Y \Lambda \transpose{Y}$, where $\Lambda$ is diagonal. By
positive semidefiniteness, $\Lambda_{ii}\ge 0$ for each $i$, so
$\sqrt{Y\Lambda}$ exists. Hence
$M=\sqrt{Y\Lambda}\transpose{(\sqrt{Y\Lambda})}$, which makes $M$ the
Gram matrix of the vector $\sqrt{Y\Lambda}$. This concludes the proof
of Thm.~\ref{t:schoenberg}.

\subsection{Finding the realization of a Gramian}
Having computed the Gram matrix $G$ from the EDM $D$ in
Sect.~\ref{s:schthm}, we obtain the corresponding
realization\index{realization} $x$ as follows. This is essentially the
same reasoning used above to show the equivalence of Gramians and PSD
matrices, but we give a few more details.

Let $\Lambda=\mbox{\sf diag}(\lambda_1,\ldots,\lambda_r)$ be the
$r\times r$ matrix with the eigenvalues
$\lambda_1\ge\ldots\ge\lambda_r$ along the diagonal and zeroes
everywhere else, and let $Y$ be the $n\times r$ matrix having the
eigenvector corresponding to the eigenvalue $\lambda_j$ as its $j$-th
column (for $j\le r$), chosen so that $Y$ consists of orthogonal
columns. Then $G=Y\Lambda\transpose{Y}$. Since $\Lambda$ is
a diagonal matrix and all its diagonal entries are nonnegative (by
positive semidefiniteness of $G$), we can write $\Lambda$ as
$\sqrt{\Lambda}\sqrt{\Lambda}$, where $\sqrt{\Lambda}=\mbox{\sf
  diag}(\sqrt{\lambda_1},\ldots,\sqrt{\lambda_r})$. Now, since
$G=x\transpose{x}$,
\begin{equation*}
x\transpose{x} =
  (Y\sqrt{\Lambda})(\sqrt{\Lambda}\transpose{Y}), 
\end{equation*}
which implies that 
\begin{equation}
x = Y\sqrt{\Lambda} \label{eq:mds13}
\end{equation}
is a realization of $G$ in $\mathbb{R}^r$. 

\subsection{Multidimensional Scaling}
MDS can be used to find realizations of approximate distance matrices
$\tilde{D}$. As above, we compute $\tilde{G}=-\frac{1}{2} J\tilde{D}^2
J$. Since $\tilde{D}$ is not a EDM, $\tilde{G}$ will probably fail to
be a Gram matrix, and as such might have negative eigenvalues. But it
suffices to let $Y$ be the eigenvectors corresponding to the $H$
positive eigenvalues $\lambda_1,\ldots,\lambda_H$, to recover an
approximate realization $x$ of $\tilde{D}$ in $\mathbb{R}^H$.

Another interesting feature of MDS is that the dimensionality $H$ of
the ambient space of $x$ is actually determined by $D$ (or
$\tilde{D}$) rather than given as a problem input. In other words, MDS
finds the ``inherent dimensionality'' of a set of (approximate)
pairwise distances.

\section{Conclusion}
\label{s:conc}
We presented what we feel are the most important and/or beautiful
theorems in DG (Heron's, Cauchy's, Cayley's, Menger's, G\"odel's and
Schoenberg's). Three of them (Heron's, Cayley's, Menger's) have to do
with the volume of simplices given its side lengths, which appears to
be {\it the} central concept in DG. We think Cauchy's proof is as
beautiful as a piece of classical art, whereas G\"odel's proof, though
less important, is stunning. Last but not least, Schoenberg's theorem
is the fundamental link between the history of DG and its contemporary
treatment.

\section*{Acknowledgments}
The first author (LL) worked on this paper whilst working at IBM TJ
Watson Research Center, and is very grateful to IBM for the freedom he
was afforded. The second author (CL) is grateful to the Brazilian
research agencies FAPESP and CNPq.


\bibliographystyle{plain} \bibliography{ibmer}

\end{document}